\documentclass{article}
\usepackage{bbm,amsmath,amssymb,amsthm,amsfonts,tikz,tikz-cd}
\usepackage[english]{babel}

\newtheorem{theorem}{Theorem}
\newtheorem{lemma}[theorem]{Lemma}
\newtheorem{definition}[theorem]{Definition}
\newtheorem{corollary}[theorem]{Corollary}

\newcommand{\RR}{\mathbb{R}}

\newcommand{\ZZ}{\mathbb{Z}}

\DeclareMathOperator{\im}{Im}
\DeclareMathOperator{\Hom}{Hom}
\begin{document}
\title{A local Jordan-Brouwer separation theorem}
\author{Alexander Lemmens}
\maketitle
\begin{abstract}
We intoduce a local version of the Jordan-Brouwer separation theorem and deduce some global statements, some of which may follow from known results, but the technique is new.
\end{abstract}
\tableofcontents
\section{Introduction}
Let $n>1$ be an integer. The following is a classical theorem in topology called the Jordan-Brouwer separation theorem:
\begin{theorem}
Let $S$ be a subset of $\RR^n$ homeomorphic to the $n-1$-dimensional sphere $S^{n-1}$, then $\RR^n\backslash S$ consists of exactly two connected components.
\end{theorem}
See \cite[theorem 3.6, p.\ 162]{dold} for a generalization of the Jordan-Brouwer se-paration theorem to manifolds. In this document we will prove a local version of the separation theorem and deduce some global results:
\begin{theorem}\label{localsep}
Let $S$ be a closed subset of $\RR^n$ containing a point $x$ that is locally around $x$ homeomorphic to an open subset of $\RR^{n-1}$, then there exist two open balls $B_2\subset B_1$ centered at $x$ such that for every connected component $C$ of $B_1\backslash S$ one of the following things holds:
\begin{itemize}
\item[$(1)$] $C$ is disjoint with $B_2$
\item[$(2)$] $B_2\cap S\subset\overline{C}$
\end{itemize}
and such that exactly two connected components of $B_1\backslash S$ satisfy $(2)$.
$B_1$ and $B_2$ can be chosen arbitrarily small. More precisely there exists $\epsilon_1>0$ such that for every $0<\delta_1<\epsilon_1$ there exists $\epsilon_2>0$ such that for every $0<\delta_2<\epsilon_2$ the balls $B_1,B_2$ centered at $x$ of radii $\delta_1,\delta_2$ respectively satisfy the above.
\end{theorem}
The intuitive meaning is that locally around $x$ the subset $S$ divides the space into two regions, corresponding to the connected components that satisfy $(2)$. Any point in $B_2\backslash S$ is contained in one of these connected components but not both (otherwise they would be one and the same connected component). Any two points around $x$ that are on the same side can be connected inside the bigger ball $B_1$ without going through $S$.\\

Of course (1) and (2) can't both be the case at the same time, because if $B_2$
is disjoint with $C$ then it is also disjoint with $\overline{C}$ as $B_2$ is open. But if (2) were
also the case then $B_2\cap S\subset\overline{C}\cap B_2 = \emptyset$, which is impossible since $x\in B_2 \cap S$.\\

From this local statement one can prove the following global result, which may already known.
\begin{corollary}\label{global}
Let $X$ be an $n$-dimensional topological manifold and $Y$ an $n-1$-dimensional submanifold, then for any connected component $C$ of $X\backslash Y$ we have that $\overline{C}\cap Y$ is a union of connected components of $Y$. If $X$ and $Y$ are both connected then $X\backslash Y$ has at most two connected components.
\end{corollary}
\begin{proof}
We have to prove that $\overline{C}\cap Y$ is both open and closed for the induced topology on $Y$. Since we take the intersection of a closed set with $Y$ the result is clearly closed for the induced topology. To prove that it is open let $x$ be any point of $\overline{C}\cap Y$, we will show that $x$ is an interior point of $\overline{C}\cap Y$ for the induced topology. To this end we take a local coordinate system around $x$ and apply theorem \ref{localsep} to obtain open balls $B_1,B_2$ centred at $x$ such that any connected component of $B_1\backslash Y$ is either disjoint with $B_2$ or has $B_2\cap Y$ in its closure. Since $x$ is in the closure of $C$ we know that $C$ contains some point of $B_2$ and therefore $C\cap B_1$ contains some connected component of $B_1\backslash Y$ that is not disjoint with $B_2$. Since this connected component has $B_2\cap Y$ in its closure we have $B_2\cap Y\subset\overline{C}$, proving that $x$ is an interior point of $\overline{C}\cap Y$ for the induced topology on $Y$.\newline Now assume $X$ and $Y$ are both connected. Then it follows that any connected component of $X\backslash Y$ has $Y$ in its closure. Let $x$ be a point of $Y$ and let $B_1$ and $B_2$ be as in theorem \ref{localsep} (with respect to some local coordinate system around $x$). We have an obvious map from the set of connected components of $B_1\backslash Y$ not disjoint with $B_2$ to the set of connected components of $X\backslash Y$ (as $B_1\backslash Y\subset X\backslash Y$). The domain has only two elements by our choice of $B_1,B_2$. The map is surjective because any connected component $C$ of $X\backslash Y$ contains some point of $B_2$ (as $x\in\overline{C}$) which must be contained in some connected component of $B_1\backslash Y$ not disjoint with $B_2$. It follows that $X\backslash Y$ has at most two connected components, as the set of these connected components is the image of a map whose domain has only two elements.
\end{proof}
We have the following extension of the last corollary:
\begin{theorem}\label{twocomponents}
In the context of corollary \ref{global} if $X$ and $Y$ are both connected (and non-empty) and $Y$ is a closed subset and $\Hom(H_1(X),\ZZ/2\ZZ)=0$, then $X\backslash Y$ has exactly two connected components.
\end{theorem}
We will prove this in section \ref{proof}. (Note that if $H_1(X)$ is finite of odd order then the requirement $\Hom(H_1(X),\ZZ/2\ZZ)=0$ is always satisfied.) From this global result we can deduce a stronger version of our local separation theorem
\begin{theorem}
Let $B_1$ be an open ball in $\RR^n$ and $S$ a subset of $B_1$ that is closed for the induced topology on $B_1$, and homeomorphic to an open subset of $\RR^{n-1}$. Let $B_2\subset B_1$ be an open ball that has nonempty intersection with exactly one connected component of $S$, then for any connected component $C$ of $B_1\backslash S$ one of the following things holds:
\begin{itemize}
\item[$(1)$] $C$ is disjoint with $B_2$
\item[$(2)$] $B_2\cap S\subset\overline{C}$
\end{itemize}
and exactly two connected components of $B_1\backslash S$ satisfy $(2)$.
\end{theorem}
\begin{proof}
Let $Y$ be the connected component of $S$ that is not disjoint with $B_2$. By the previous theorem $B_1\backslash Y$ has exactly two connected components, call them $C_1$ and $C_2$. Let
$$Y'=S\backslash Y=S\cap(C_1\cup C_2),$$
then the connected components of $B_1\backslash S$ are exactly the connected components of $C_1\backslash Y'$ and $C_2\backslash Y'$. Let us show for $i=1,2$ that $C_i\cap B_2$ is contained in a single connected component of $B_1\backslash S$. Note that $C_i\cap B_2\subset C_i\backslash Y'$ because $B_2$ is disjoint with $Y'$. So let $x$ and $y$ be points of $C_i\cap B_2$, consider the Mayer Vietoris sequence
$$\ldots\longrightarrow H_1(B_1)\longrightarrow H_0(B_1\backslash S)\longrightarrow H_0(B_1\backslash Y)\oplus H_0(B_1\backslash Y')\longrightarrow\ldots.$$
Of course $H_1(B_1)=0$. Now $x-y$ is defines an element of $H_0(B_1\backslash S)$, and this element becomes zero in $H_0(B_1\backslash Y)$ (because $x$ and $y$ belong to the same connected component of $B_1\backslash Y$) and in $H_0(B_1\backslash Y')$ (because $x$ and $y$ both belong to $B_2$ which is a connected subset of $B_1\backslash Y'$). Therefore, by the exactness of the Mayer-Vietoris sequence $x-y$ is zero in $H_0(B_1\backslash S)$. In other words they are contained in a single connected component of $B_1\backslash S$.\\

This means there are exactly two connected components of $B_1\backslash S$ that are not disjoint with $B_2$. One of them contains $C_1\cap B_2$, the other contains $C_2\cap B_2$. By corollary \ref{global} each $C_i$ has $Y$ in its closure, so $B_2\cap C_i$ has $B_2\cap Y=B_2\cap S$ in its closure, proving the theorem.
\end{proof}
We close this section with another corollary of the local separation theorem.
\begin{corollary}
In the situation of theorem \ref{twocomponents} if $y\in Y$ then $(X\backslash Y)\cup\{y\}$ is pathwise connected.
\end{corollary}
\begin{proof}
We take balls $B_{1,2}\subset B_{1,1}\subset X$ centred at $y$ (in some local coordinate system) such that every connected component of $B_{1,1}\backslash Y$ satisfies $(1)$ or $(2)$ and two of them satisfy $(2)$.
Let $x_0$ and $z_0$ be points of $B_{1,1}\backslash Y$ belonging to different connected components of $X\backslash Y$. We have to prove that they can be connected by a path that intersects $Y$ only at $y$.\\

We now recursively define a sequence of pairs of balls $B_{i,2}\subset B_{i,1}$, $i\geq 1$, centered at $y$, each pair satisfying the conclusion of theorem \ref{localsep}, by repeatedly applying this theorem. We choose the balls so that $B_{i+1,1}\subset B_{i,2}$. Because the balls can be chosen arbitrarily small we can make sure that the intersection of all of them is $\{y\}$.\\

Next we inductively construct two sequences of points $x_i$, $z_i$, $i\geq 0$ with $x_i,z_i\in B_{i+1,2}\backslash Y$ such that for each $i\geq 0$ we have that $x_i$ and $x_{i+1}$ (resp.\ $z_i$ and $z_{i+1}$) are in the same connected component of $B_{i+1,1}\backslash Y$. Suppose some $x_i$ has already been constructed. To construct $x_{i+1}$, let $U_i$ be the connected component of $B_{i+1,1}\backslash Y$ containing $x_i$. Then $U_i\cap B_{i+1,2}\neq\emptyset$ and so $Y\cap B_{i+1,2}\subset\overline{U_i}$. Since $y$ is a closure point of $U_i$, we know that $U_i$ cannot be disjoint with $B_{i+2,2}$. Therefore we can take a point $x_{i+1}$ in $U_i\cap B_{i+2,2}$. We have a similar argument for the $z_i$.\\

By construction both sequences of points converge to $y$. Now for each $i\geq 1$ we can take a path $\gamma_{i,1}$ from $x_{i-1}$ to $x_i$ and a path $\gamma_{i,2}$ from $z_{i-1}$ to $z_i$, where the image of both paths is contained in $B_{i,1}\backslash Y$.
Finally one joins all the $\gamma_{i,1}$ into a path from $x_0$ to $y$ and the $\gamma_{i,2}$ into a path from $z_0$ to $y$, which together gives a path from $x_0$ to $z_0$ intersecting $Y$ only at $y$.
\end{proof}
\section{Proof of local separation theorem}
In this section we prove the local separation theorem from scratch. All we use in this proof is reduced homology, the concept of inverse limits of Abelian groups and the fact that the complement of a closed ball embedded in a sphere has trivial reduced homology.
\begin{lemma}\label{zeromap}
Let $S$ be a subset of $\RR^n$ and $f$ a homeomorphism from a closed ball $D$ in $\RR^{n-1}$ onto $S$. Let $B_1$ be an open ball in $\RR^n$ with $B_1\cap S\subset f\Big(\overset{\circ}{D}\Big)$, let $D'$ be a closed ball contained in $f^{-1}(B_1)$ and let $B_2\subset B_1$ be an open ball with $B_2\cap S\subset f(D')$. Let $L$ be a linear subspace in $\RR^{n-1}$ of dimension at least one containing some interior point $p$ of $D'$. Let $H$ be some half space of $L$ whose boundary in $L$ contains $p$. Then for each integer $i$ the morphisms
$$\overset{\sim}{H}_i(B_2\backslash f(H\cap D))\longrightarrow\overset{\sim}{H}_i(B_1\backslash f(H\cap D))$$
of reduced homology are zero.
\end{lemma}
\begin{proof}
Let $X$ be a closed line segment from $H\cap D'$ to $f^{-1}(\partial B_1)$ having only one point in common with $f^{-1}(\partial B_1)$ and $D'$ respectively. Let
$$C=B_1\backslash f(X\cup(H\cap D')),$$
we claim that all of the reduced homology of $C$ is zero. To prove this we will use the result that for a sphere minus an embedded disk all of the reduced homology is zero. We also use the result that if $A$ and $B$ are open sets in a topological space such that $A,B$ and $A\cup B$ all have the property that all reduced homology is zero then $A\cap B$ also has this property. This follows immediately from the Mayer Vietoris sequence.\newline
If we take the one point compactification $\RR^n\cup\{\infty\}$ of $\RR^n$ we get something homeomorphic to a sphere. We know that all of the homology of $\RR^n\cup\{\infty\}\backslash f(X)$ is zero, and we know the same for
$\RR^n\cup\{\infty\}\backslash f(H\cap D')$ because $H\cap D'$ is homeomorphic to a disk. We also know it for the union of these two sets, namely $\RR^n\cup\{\infty\}\backslash f(X\cap H\cap D')$ because $X\cap H\cap D'$ is just one point, therefore we know it for the intersection $\RR^n\cup\{\infty\}\backslash f(X\cup (H\cap D'))$. Since it is also true for $B_1$ and for the union
$$B_1\cup\big(\RR^n\cup\{\infty\}\backslash f(X\cup (H\cap D'))\big)$$
which is $\RR^n\cup\{\infty\}$ minus one point, it is also true for the intersection $C$, proving the claim. Let
$$C'=B_1\backslash f(X\cup(H\cap\partial D'))$$
then by a similar argument the reduced homology of $C'$ is also zero. Let
$$E=B_1\backslash f\Big(D\cap H\backslash\overset{\circ}{D'}\Big),$$
then $C'=C\cup E$ and $C\cap E=B_1\backslash f(D\cap H)$. The associated Mayer Vietoris sequence is
$$\ldots\longrightarrow\overset{\sim}{H}_{i+1}(C')\longrightarrow \overset{\sim}{H}_i(C\cap E)\longrightarrow \overset{\sim}{H}_i(C)\oplus \overset{\sim}{H}_i(E)\longrightarrow\ldots,$$
but we proved that all homology of $C$ and $C'$ are zero, so we get an isomorphism from $\overset{\sim}{H}_i(B_1\backslash f(D\cap H))$ to $\overset{\sim}{H}_i(E)$. Now the morphism from $\overset{\sim}{H}_i(B_2\backslash f(H\cap D))$ to $\overset{\sim}{H}_i(E)$ is zero because it factors through $\overset{\sim}{H}_i(B_2)=0$ (since $B_2\subset E$). It follows that the morphism from $\overset{\sim}{H}_i(B_2\backslash f(H\cap D))$ to $\overset{\sim}{H}_i(B_1\backslash f(H\cap D))$ is zero, as its composition with the isomorphism from $\overset{\sim}{H}_i(B_1\backslash f(H\cap D))$ to $\overset{\sim}{H}_i(E)$ is zero.
\end{proof}
\begin{definition}
Let
$$\ldots\longrightarrow A_n\longrightarrow \ldots \longrightarrow A_2 \longrightarrow A_1 \longrightarrow A_0$$
be a sequence of Abelian groups with morphisms. If $a,b\in\ZZ_{\geq 0}$ we will say the sequence is $a,b$-injective if for all $c\geq a$ and $x\in A_{b+c}$ if $x$ becomes zero in $A_a$ then $x$ becomes zero in $A_c$. We will say the sequence is $a,b$-surjective if for all $c\geq a$ and $x\in A_c$ if $x$ is reached by $A_{b+c}$ then $x$ is reached by $A_d$ for all $d\geq c$.
\end{definition}
\begin{lemma}
In the situation of last definition we have:
\begin{itemize}\label{iii}
\item[$(i)$] If the sequence is $a,b$-injective (resp.\ $a,b$-surjective) then it is $a',b'$-injective (resp.\ $a',b'$-surjective) for all $a'\geq a$, $b'\geq b$.
\item[$(ii)$] If all the morphisms $A_{i+1}\rightarrow A_i$ are zero then the sequence is $0,1$-injective and $0,1$-surjective.
\item[$(iii)$] If all the morphisms $A_{i+1}\rightarrow A_i$ are isomorphisms then the sequence is $0,0$-injective and $0,0$-surjective.
\item[$(iv)$] If the sequence is both $a,d$-injective and $a,b$-surjective then for all $c\geq a$ the natural map
$$\lim_{\longleftarrow}A_i\longrightarrow\im(A_{c+b}\rightarrow A_{c})$$
is an isomorphism.
\end{itemize}
\end{lemma}
\begin{proof}
The proof of $(i),(ii)$ and $(iii)$ is left as an exercise to the reader. As for $(iv)$ let us first define the natural map. An element of $\lim_{\longleftarrow}A_i$ is given by a sequence $(x_i)_{i\geq 0}$ with $x_i\in A_i$ where each $x_{i+1}$ gets mapped to $x_i$. The natural map is defined by mapping this sequence to $x_c$, which will be in the image of the map from $A_{b+c}$ to $A_c$.\newline
We now prove injectivity of the natural map. Suppose $x_c$ is zero, we have to prove that all the $x_i$ are zero. For $i\leq c$ this is clear so let $i>c$. Since $x_{d+i}$ gets mapped to zero in $A_a$ it must be mapped to zero in $A_i$, by $a,d$-injectivity, so $x_i=0$.\newline
We now prove surjectiviy of the natural map. Let $x\in A_c$ such that $x$ is in the image of the map from $A_{b+c}$ to $A_c$. By definition of $a,b$-surjective it is also in the image of $A_d$ for all $d\geq c$. Let $b'=\max(b,d)$. We claim that for each $i\geq c$ there is a unique element $x_i$ of $A_i$ that gets reached by $A_{b'+i}$ and gets mapped to $x$ in $A_c$. The existence follows from the above, suppose we have two such elements $x_i$ and $x'_i$ then $x_i-x'_i$ gets mapped to zero in $A_c$ while being reached by $A_{b'+i}$ and hence by $A_{d+i}$. By $a,d$-injectivity $x_i=x'_i$, proving uniqueness. This allows us to construct an element $(x_i)_{i\geq 0}$ of $\lim_{\longleftarrow}A_i$ that gets mapped to $x$ as clearly $x_c=x$.
\end{proof}
The idea will be to use lemma \ref{zeromap} to create a situation where $(ii)$ of lemma \ref{iii} happens, and use Mayer Vietoris sequences in combination with the next lemma to imitate the proof of the Jordan separation theorem.
\begin{lemma}\label{five}
Suppose we have a commutative diagram
\begin{center}
\begin{tikzcd}
\vdots &\vdots &\vdots &\vdots &\vdots\\
A_i\arrow{r}\arrow{u}&B_i\arrow{r}\arrow{u}&C_i\arrow{r}\arrow{u}&D_i\arrow{r}\arrow{u}&E_i\arrow{u}\\
A_{i+1}\arrow{r}\arrow{u}&B_{i+1}\arrow{r}\arrow{u}&C_{i+1}\arrow{r}\arrow{u}&D_{i+1}\arrow{r}\arrow{u}&E_{i+1}\arrow{u}\\
\vdots\arrow{u} &\vdots\arrow{u} &\vdots\arrow{u} &\vdots\arrow{u} &\vdots\arrow{u}\\
\end{tikzcd}
\end{center}
where $i$ runs through $\ZZ_{\geq 0}$ and suppose the rows are exact.\begin{itemize}
\item[$(i)$] If $(A_i)_{i\geq 0}$ is $a,b$-surjective, $(B_i)_{i\geq 0}$ is $a,c$-injective and $(D_i)_{i\geq 0}$ is $a+b,d$-injective then $(C_i)_{i\geq 0}$ is $a+b,c+d$-injective.
\item[$(ii)$] If $(B_i)_{i\geq 0}$ is $a,b$-surjective, $(D_i)_{i\geq 0}$ is $a+b,c$-surjective and $(E_i)_{i\geq 0}$ is $a+b,d$-injective then $(C_i)_{i\geq 0}$ is $a,b+c$-surjective.
\item[$(iii)$] If $(B_i)_{i\geq 0}$ and $(C_i)_{i\geq 0}$ are $a,b$-injective and $a,b$-surjective for whatever $a,b$ then the sequence
$$\lim_{\longleftarrow}B_i\longrightarrow\lim_{\longleftarrow}C_i\longrightarrow \lim_{\longleftarrow}D_i$$
is exact.
\end{itemize}
\end{lemma}
Note the formal similarity of $(i)$ and $(ii)$ with five lemma.
\begin{proof}
$(i)$ Let $e\geq a+b$ and $x\in C_{c+d+e}$ and suppose $x$ becomes zero in $C_{a+b}$, we have to show it becomes zero in $C_e$. Let us denote the image of $x$ in $C_i$ (resp.\ $D_i$) for $i\leq c+d+e$ by $x_{C,i}$ (resp.\ $x_{D,i}$). As $x_{C,a+b}=0$ we have $x_{D,a+b}=0$ and hence $x_{D,c+e}=0$ by the $a+b,d$-injectivity of $(D_i)_{i\geq 0}$. By exactness of the rows $x_{C,c+e}$ is reached by some $x_{B,c+e}\in B_{c+e}$. Because $x_{B,a+b}$ becomes zero in $C_{a+b}$ it is reached by some $x_{A,a+b}$, again by exactness. Denote its image in $A_{a}$ by $x_{A,a}$, then $x_{A,a}$ is reached by some $x_{A,c+e}$ by $a,b$-surjectivity of $(A_i)_{i\geq 0}$. Note that the image of $x_{A,c+e}$ in $A_{a+b}$ is not necessarily $x_{A,a+b}$. Denote the image of $x_{A,c+e}$ in $B_{c+e}$ by $x'_{B,c+e}$. Then $x_{B,c+e}-x'_{B,c+e}$ gets mapped to zero in $B_{a}$, so $x_{B,e}=x'_{B,e}$ by $a,c$-injectivity of $(B_i)_{i\geq 0}$. By exactness $x_{B,e}$ gets mapped to zero in $C_{e}$ as it is reached by $A_e$. So $x_{C,e}=0$, concluding the proof of $(i)$.\\

\noindent $(ii)$ Let $e\geq a$ and $x\in C_e$ so that $x$ gets reached by $C_{b+c+e}$, we have to prove that it gets reached by $C_{f}$ for all $f\geq e$. This is clear when $f\leq b+c+e$, so assume $f>b+c+e$. Let $x_{C,b+e}$ be an element of $C_{b+e}$ that gets mapped to $x=x_{C,e}$. Then $x_{C,b+e}$ gets reached by $C_{b+c+e}$ and so $x_{D,b+e}$ gets reached by $D_{b+c+e}$. By $a+b,c$-surjectivity of $(D_i)_{i\geq 0}$ $x_{D,b+e}$ gets reached by some $x_{D,d+f}\in D_{d+f}$. By exactness $x_{E,b+e}=0$ as it gets reached by $x_{C,b+e}$. By $a+b,d$-injectivity of $(E_i)_{i\geq 0}$ we have $x_{E,f}=0$. By exactness $x_{D,f}$ gets reached by some $x_{C,f}$. But this $x_{C,f}$ does not necessarily get mapped to $x_{C,e}=x$. Denote its image in $C_{b+e}$ by $x'_{C,b+e}$. Then by exactness $x_{C,b+e}-x'_{C,b+e}$ gets reached by some $x_{B,b+e}\in B_{b+e}$. Then by $a,b$-surjectivity $x_{B,e}$ gets reached by some $x_{B,f}$. Denote its image in $C_{f}$ by $x'_{C,f}$. This gets mapped to $x_{C,e}-x'_{C,e}$. Therefore $x=x_{C,e}$ gets reached by $x_{C,f}+x'_{C,f}$.\\

\noindent $(iii)$ Let $(x_i)_{i\geq 0}$ be an element of $\lim_{\longleftarrow}C_i$ that gets mapped to zero in $\lim_{\longleftarrow}D_i$. Then by exactness each $x_i$ gets reached by some element of $B_i$. Let $y$ be an element of $B_{a+b}$ that gets mapped to $x_{a+b}$ and let $y'$ be its image in $B_a$. By Lemma \ref{iii} $(iv)$ there exists a $(y_i)_{i\geq 0}\in \lim_{\longleftarrow}B_i$ with $y_a=y'$. Here we use the surjectivity of the natural map. Let $(x'_i)_{i\geq 0}$ be its image in $\lim_{\longleftarrow}C_i$, we claim that $x_i=x'_i$ and then we are done. Let $x''_i=x_i-x'_i$, then $x''_a=0$. Again by lemma \ref{iii} $(iv)$ we have $x''_i=0$ for all $i$. Here we use the injectivity of the natural map.
\end{proof}
Now we can use these results to mimic the proof of the usual Jordan separation theorem.
\begin{lemma}\label{n-1}
Let $S$ be a subset of $\RR^n$ and $f$ a homeomorphism from a closed ball $D$ in $\RR^{n-1}$ onto $S$. Suppose we have an infinite sequence of open balls $(B_i)_{i\geq 0}$ in $\RR^n$, each one containing the next one in the sequence, and all containing some point $p\in S$, such that for each $i\geq 0$ there exists a closed ball contained in $f^{-1}(B_i)$ whose image contains $B_{i+1}\cap S$. Also assume that $B_0\cap S$ is contained in $f\Big(\overset{\circ}{D}\Big)$. Then for all $j$ the sequence
$$\ldots\longrightarrow \overset{\sim}{H}_j(B_i\backslash S)\longrightarrow \ldots\longrightarrow\overset{\sim}{H}_j(B_0\backslash S)$$
is $n-1,n-1$-injective and $0,n-1$-surjective and
$$\lim_{\longleftarrow}\overset{\sim}{H}_j(B_i\backslash S)\cong
\begin{cases}
\ZZ&\textup{if }j=0\\
0&\textup{else.}\\
\end{cases}
$$
\end{lemma}
Note that the balls do not have to be concentric, and nor does their intersection have to equal $\{p\}$, even though intuitively one expects the intersection to be $\{p\}$.
\begin{proof}
We can assume $f(0)=p$, otherwise we compose $f$ with a translation in $\RR^{n-1}$. For $0\leq j\leq n-1$ define
$$L_j=\{(x_1,\ldots,x_j,0,\ldots,0)\in D\},$$
$$H_j^+=\{(x_1,\ldots,x_j,0,\ldots,0)\in D|x_j\geq 0\}\text{ and}$$
$$H_j^-=\{(x_1,\ldots,x_j,0,\ldots,0)\in D|x_j\leq 0\}.$$
By lemma \ref{zeromap} the maps in the following sequences are all zero:
$$\ldots\longrightarrow \overset{\sim}{H}_j(B_i\backslash f(H_j^+))\longrightarrow \ldots\longrightarrow\overset{\sim}{H}_j(B_0\backslash f(H_j^+))$$
$$\ldots\longrightarrow \overset{\sim}{H}_j(B_i\backslash f(H_j^-))\longrightarrow \ldots\longrightarrow\overset{\sim}{H}_j(B_0\backslash f(H_j^-)).$$
So we have
$$\lim_{\longleftarrow}\overset{\sim}{H}_j(B_i\backslash f(H_j^+))=\lim_{\longleftarrow}\overset{\sim}{H}_j(B_i\backslash f(H_j^-))=0$$
and by lemma \ref{iii} $(ii)$ these sequences are all $0,1$-injective and $0,1$-surjective. Next note that $L_0$ is just $\{0\}$ so $B_i\backslash f(L)=B_i\backslash\{p\}$ and the reduced homology of this is zero except the $n-1$-th homology is $\ZZ$. Therefore the sequence
$$\ldots\longrightarrow \overset{\sim}{H}_j(B_i\backslash f(L_0))\longrightarrow \ldots\longrightarrow\overset{\sim}{H}_j(B_0\backslash f(L_0))$$
consists of isomorphisms, so by lemma \ref{iii} $(iii)$ it is $0,0$-injective and $0,0$-surjective and of course we have
$$\lim_{\longleftarrow}\overset{\sim}{H}_j(B_i\backslash L_0)\cong
\begin{cases}
\ZZ&\textup{if }j=n-1\\
0&\textup{else.}\\
\end{cases}
$$
Our strategy is to prove by induction on $j'$ that the sequence
$$\ldots\longrightarrow \overset{\sim}{H}_j(B_i\backslash f(L_{j'}))\longrightarrow \ldots\longrightarrow\overset{\sim}{H}_j(B_0\backslash f(L_{j'}))$$
is $j',j'$-injective and $0,j'$-surjective and that
$$\lim_{\longleftarrow}\overset{\sim}{H}_j(B_i\backslash L_{j'})\cong
\begin{cases}
\ZZ&\textup{if }j=n-1-j'\\
0&\textup{else.}\\
\end{cases}
$$
For $j'=n-1$ this gives the desired statement as $f(L_{n-1})=f(D)=S$. We already did the induction basis $j'=0$. Now for the induction step suppose it is true for some $j'$, we prove it for $j'+1$. For any $i\geq 0$ we have that
$$B_i\backslash f(L_{j'})=B_i\backslash f(H^+_{j'+1})\cup B_i\backslash f(H^-_{j'+1})\text{ and}$$
$$B_i\backslash f(L_{j'+1})=B_i\backslash f(H^+_{j'+1})\cap B_i\backslash f(H^-_{j'+1}).$$
Since these are open sets we have a Mayer-Vietoris sequence
$$\ldots\longrightarrow \overset{\sim}{H}_{j+1}(B_i\backslash f(H^+_{j'+1}))\oplus\overset{\sim}{H}_{j+1}(B_i\backslash f(H^{-}_{j'+1}))\longrightarrow\overset{\sim}{H}_{j+1}(B_i\backslash f(L_{j'}))$$
$$\longrightarrow\overset{\sim}{H}_j(B_i\backslash f(L_{j'+1}))\longrightarrow\overset{\sim}{H}_j(B_i\backslash f(H^+_{j'+1}))\oplus\overset{\sim}{H}_j(B_i\backslash f(H^{-}_{j'+1}))\longrightarrow\ldots$$
and we are in a situation where we can apply lemma \ref{five}, with
\begin{align}
&A_i=\overset{\sim}{H}_{j+1}(B_i\backslash f(H^+_{j'+1}))\oplus\overset{\sim}{H}_{j+1}(B_i\backslash f(H^{-}_{j'+1}))\nonumber \\
&B_i=\overset{\sim}{H}_{j+1}(B_i\backslash f(L_{j'}))\nonumber \\
&C_i=\overset{\sim}{H}_j(B_i\backslash f(L_{j'+1}))\nonumber \\
&D_i=\overset{\sim}{H}_j(B_i\backslash f(H^+_{j'+1}))\oplus\overset{\sim}{H}_j(B_i\backslash f(H^{-}_{j'+1}))\nonumber \\
&E_i=\overset{\sim}{H}_j(B_i\backslash f(L_{j'+1}))\nonumber
\end{align}
So $(A_i)_{i\geq 0}$ and $(D_i)_{i\geq 0}$ are $0,1$-injective and $0,1$-surjective and by induction hypothesis $(B_i)_{i\geq 0}$ and $(E_i)_{i\geq 0}$ are $j',j'$-injective and $0,j'$-surjective. By lemma \ref{five} $(i)$ and $(ii)$ (and lemma \ref{iii} $(i)$) $(C_i)_{i\geq 0}$ is $j'+1,j'+1$-injective and $0,j'+1$-surjective. By applying $(iii)$ of the same lemma we get a long exact sequence of inverse limits
$$\ldots\longrightarrow\lim_{\longleftarrow}A_i\longrightarrow\lim_{\longleftarrow}B_i\longrightarrow\lim_{\longleftarrow}C_i\longrightarrow \lim_{\longleftarrow}D_i\longrightarrow\ldots.$$
But $\lim_{\longleftarrow}A_i=\lim_{\longleftarrow}D_i=0$ so we get isomorphisms
$$\lim_{\longleftarrow}\overset{\sim}{H}_{j+1}(B_i\backslash f(L_{j'}))\longrightarrow\lim_{\longleftarrow}\overset{\sim}{H}_j(B_i\backslash f(L_{j'+1})).$$
By induction hypothesis the left hand side is $\ZZ$ if $j+1=n-1-j'$ and zero otherwise. Therefore the right hand side is $\ZZ$ if $j=n-1-(j'+1)$ and zero otherwise, finishing the proof.
\end{proof}
\begin{proof}[Proof of theorem \ref{localsep}]
We can assume that $S$ is the image of a homeomorphism $f$ from a closed ball $D$ in $\RR^{n-1}$.
Let us call a pair of balls $B'_1,B'_2$ in $\RR^n$ `good' if $B_2'\subset B_1'$ and there exists a closed ball $D'$ in $f^{-1}(B'_1)$ whose image contains $B'_2\cap S$.
We will prove that for any sequence $B'_0,\ldots,B'_{2n-2}$ of open balls centred at $x$ where any two consecutive balls are a good pair, and where $B'_0\cap S\subset f\Big(\overset{\circ}{D}\Big)$, the balls $B_1:=B'_{n-1}$ and $B_2:=B'_{2n-2}$ satisfy the theorem. It is enough to show for any point $p\in B_2\cap S$ that any connected component of $B_1\backslash S$ is either disjoint with $B_2$ or has $p$ in its closure and there are exactly two such components with $p$ in their closure. To this end choose balls $B'_i$, $i>2n-2$ such that each pair $B'_i,B'_{i+1}$, $i\geq 0$ is good, and such that the intersection of all these balls is $\{p\}$. By lemma \ref{n-1} the sequence
$$\ldots\longrightarrow \overset{\sim}{H}_0(B'_i\backslash S)\longrightarrow \ldots\longrightarrow\overset{\sim}{H}_0(B'_0\backslash S)$$
is $n-1,n-1$-injective and $0,n-1$-surjective and the inverse limit is isomorphic to $\ZZ$. By lemma \ref{iii} $(iv)$ (and $(i)$) the image of the map
$$\alpha_i:\overset{\sim}{H}_0(B'_i\backslash S)\longrightarrow\overset{\sim}{H}_0(B'_{n-1}\backslash S)$$
is isomorphic to $\ZZ$ for all $i\geq 2n-2$. For any $i\geq 0$ let $C_i$ be the set of all connected components of $B'_i\backslash S$, then
$$\overset{\sim}{H}_0(B'_i\backslash S)\cong\ker\Big(\bigoplus_{C_i}\ZZ\longrightarrow\ZZ\Big),$$
where the map $\bigoplus_{C_i}\ZZ\longrightarrow\ZZ$ maps any element of any copy of $\ZZ$ to itself in $\ZZ$. The image of the map $\alpha_i$ is then
$$\im\alpha_i\cong\ker\Big(\bigoplus_{\im(C_i\rightarrow C_{n-1})}\ZZ\longrightarrow\ZZ\Big).$$
The direct sum is over all connected components of $B'_{n-1}\backslash S$ that contain some connected component of $B'_i\backslash S$. So the fact that for all $i\geq 2n-2$ the image of $\alpha_i$ is isomorphic to $\ZZ$ means that exactly two connected components of $B'_{n-1}\backslash S$ contain some connected component of $B'_i\backslash S$. Of course those two connected components of $B'_{n-1}\backslash S$ will not depend on the choice of $i\geq 2n-2$. Since the intersection of the $B'_i$ is $\{p\}$ these two connected components must have $p$ in their closure. Because only these two connected components of $B'_{n-1}\backslash S$ contain some connected component of $B'_{2n-2}\backslash S$, all other connected components of $B'_{n-1}\backslash S$ are disjoint with all connected components of $B'_{2n-2}\backslash S$, and are therefore disjoint with $B'_{2n-2}$. This concludes the proof.
\end{proof}
\section{Proof of theorem \ref{twocomponents}}\label{proof}
In this section we use the local separation theorem (theorem \ref{localsep}) to prove theorem \ref{twocomponents}.
\begin{lemma}\label{agree}
Let $X$ be an $n$-dimensional topological manifold and let $Y$ be a closed subset that is an $n-1$-dimensional submanifold. Suppose $B_1,B_2$ and $B'_1,B'_2$ are pairs of open balls in $X$ with respect to two possibly different coordinate systems, and suppose they each satisfy the condition in theorem \ref{localsep}. If $U$ is a connected component of $B_1\cap B_2$ then any two points $x,y\in U\backslash Y$ belong to the same connected component of $B_1\backslash Y$ if and only if they belong to the same connected component of $B'_1\backslash Y$.
\end{lemma}
Intuitively this lemma says that different pairs of balls like in theorem \ref{localsep} locally agree on whether or not two points are on the same side of $Y$.
\begin{proof}
We first prove that any point $p\in U\cap Y$ has a neighbourhood $V$ such that any two points $x,y\in V\backslash Y$ belong to the same connected component of $B_1\backslash Y$ if and only if they belong to the same connected component of $B'_1\backslash Y$. This follows by applying theorem \ref{localsep} to obtain $B''_1,B''_2$ with
$$p\in B''_2\subset B''_1\subset U$$
and taking $V=B''_2$. Let us explain why this works. We have a map $\iota$ (resp.\ $\iota'$) from the set of connected components of $B''_1\backslash Y$ that are not disjoint with $B''_2$ to the set of connected components of $B_1\backslash Y$ (resp.\ $B'_1\backslash Y$) that are not disjoint with $B_2$ (resp.\ $B'_2$). These maps are surjective because each connected component of $B_1\backslash Y$ (resp.\ $B'_1\backslash Y$) not disjoint with $B_2$ (resp.\ $B'_2$) has $p$ in its closure, and therefore contains some point of $B''_2$. Note that the domain and image of $\iota,\iota'$ each have two elements, so they are bijections. It follows that two points of $V\backslash Y= B''_2\backslash Y$ belong to the same connected component of $B_1\backslash Y$ if and only if they belong to the same connected component of $B''_1\backslash Y$ if and only if they belong to the same connected component of $B'_1\backslash Y$.\\

Now we use this local result to prove the lemma. Let us call a pair of points in $U\backslash Y$ good if they belong to the same connected component of $B_1\backslash Y$ if and only if they belong to the same connected component of $B'_1\backslash Y$. Now if $x,y$ is a good pair of points and $y,z$ is a good pair of points then so is $x,z$. This is because there are only two connected components of $B_1\backslash Y$ (resp.\ $B'_1\backslash Y$) that $x$, $y$ and $z$ could be in. If for instance $x$ and $y$ belong to different connected components of $B_1\backslash Y$ (resp.\ $B'_1\backslash Y$) and $y$ and $z$ belong to different connected components of $B_1\backslash Y$ (resp.\ $B'_1\backslash Y$) then $x$ and $z$ must belong to the same connected component of $B_1\backslash Y$ (resp.\ $B'_1\backslash Y$). Now in fact this means that being a good pair of points is an equivalence relation on $U\backslash Y$. We have to use the connectedness of $U$ to prove that there is only one equivalence class. Since pairs of points belonging to the same connected component of $U\backslash Y$ are clearly good, we know that each equivalence class must be a union of connected components of $U\backslash Y$. Let $U'$ be an equivalence class (which must of course be open). Let $Y'=Y\cap\overline{U'}\cap U$. If we can prove that $U'\cup Y'$ is both open and closed for the induced topology on $U$ then by connectedness it must be all of $U$ and then $U'=(U'\cup Y')\backslash Y=U\backslash Y$ and we are done.\newline 
Before we do so we prove that $U'\cup Y'=\overline{U'}\cap U$. Clearly $U'\cup Y'\subset\overline{U'}\cap U$. That $\overline{U'}\cap U\subset U'\cup Y'$ is proved as follows. First note that $U'$ is closed for the induced topology on $U\backslash Y$, since it is a union of connected components, and so is its complement, and therefore its complement is open, so $U'$ is closed. Now let $x\in\overline{U'}\cap U$, if $x\in Y$ then $x\in Y\cap\overline{U'}\cap U=Y'$ so $x\in U'\cup Y'$. If $x\notin Y$ then $x\in U\backslash Y$ which together with $x\in\overline{U'}$ implies that $x\in U'$ by the closedness of $U'$ in $U\backslash Y$, so $x\in U'\cup Y'$, proving the other inclusion.\newline
That $U'\cup Y'$ is closed for the induced topology on $U$ is now clear since it equals $\overline{U'}\cap U$.\\

We now prove that $U'\cup Y'$ is open by showing that all of its points are interior points. Let $x\in U'\cup Y'$, if $x\in U'$ then it is obviously an interior point since $U'$ is open. So suppose $x\in Y'$. Then by the first part of the proof there is an open set $V\subset U$ containing $x$ such that all pairs of points in $V\backslash Y$ are good. In other words $V\backslash Y$ is contained in a single equivalence class. As $x$ is in the closure of $U'$ we know that $V\cap U'=(V\backslash Y)\cap U'$ is nonempty. Therefore, Since $U'$ is an equivalence class it must contain $V\backslash Y$. Hence $U'\cup Y'=\overline{U'}\cap U$ contains $\overline{V\backslash Y}\cap U$, which in turn contains $V$ as $Y$ has empty interior. This proves that $x$ is an interior point.
\end{proof}
\noindent\textbf{Remark}\vspace{5 pt}\newline
\noindent Note that if we take points $x,y$ in different connected components of $B_2\cap B'_2$ then the pairs $B_1,B_2$ and $B'_1,B'_2$ might disagree on whether or not $x$ and $y$ are on the same side of $Y$. As an example consider two rectangular strips $B_1=B_2$ and $B'_1=B'_2$ that we paste together to form a M\"obius strip $X$, and for $Y$ we take a line that goes around the M\"obius strip once, and we take $x$ from one piece of overlap and $y$ from the other piece of overlap, then $B_1,B_2$ and $B'_1,B'_2$ will disagree on whether $x$ and $y$ are on the same side of $Y$ or on opposite sides of it.\newpage

\noindent\textbf{Construction}\normalsize\vspace{5 pt}\newline
\noindent Let $X$ be an $n$-dimensional manifold and $Y$ an $n-1$-dimensional submanifold that is a closed subset of $X$. Let $x,y\in X\backslash Y$ and let $\gamma:[0,1]\longrightarrow X$ be a path from $x$ to $y$ (so $\gamma(0)=x$ and $\gamma(1)=y$). We will construct an element $e(\gamma)$ of $\ZZ/2\ZZ$ that captures the intuition of whether the path $\gamma$ passes through $Y$ an even or an odd number of times. It will be defined in such a way that if $\im\gamma$ is contained in some $B_2$ where $B_1,B_2$ is as in theorem \ref{localsep} then $e(\gamma)=0$ if and only if $x$ and $y$ belong to the same connected component of $B_1\backslash Y$. The construction is as follows. We choose a set $0=a_0<a_1<\ldots<a_k=1$ and for each $i=1,\ldots,k$ we choose a pair $B_{1,i},B_{2,i}$ as in theorem \ref{localsep} such that $\gamma([a_{i-1},a_i])\subset B_2$ unless $\gamma([a_{i-1},a_i])$ is disjoint with $Y$, in which case we put $B_{1,i}=B_{2,i}$ equal to some connected open set disjoint with $Y$ containing $\gamma([a_{i-1},a_i])$. Next for each $i=1,\ldots,k-1$ we choose a path $\gamma_i:[0,1]\longrightarrow B_{2,i}\cap B_{2,i+1}$ from $\gamma(a_i)$ to some point not in $Y$. We take $\gamma_0$ to be the constant path from $x$ to itself and $\gamma_k$ the constant path from $y$ to itself. For each $i=1,\ldots,k$ we define $e(a_{i-1},a_i)$ as $0\in\ZZ/2\ZZ$ if $\gamma_{i-1}(1)$ and $\gamma_i(1)$ belong to the same connected component of $B_{1,i}\backslash Y$
and as $1\in\ZZ/2\ZZ$ if they don't. Finally we define
$$e(\gamma)=e(a_0,a_1)+\ldots+e(a_{k-1},a_k).$$
\begin{lemma}
This construction is always possible and does not depend on the choice of $k,a_i,\gamma_i,B_{1,i},B_{2,i}$.
\end{lemma}
Note that the reason why we have to take $\gamma_i$ is because $\gamma(a_i)$ may be a point of $Y$, in which case we have to take a point close to $\gamma(a_i)$ that is not in $Y$.
\begin{proof}
By theorem \ref{localsep} the $B_2$ coming from that theorem cover $Y$. Therefore the inverse images $\gamma^{-1}(B_2)$ together with $\gamma^{-1}(X\backslash Y)$ yield an open cover of $[0,1]$. So we can choose $a_1,\ldots,a_{k-1}$ so that each interval $[a_{i-1},a_i]$ for $i=1,\ldots,k$ is contained in some open set of this cover. Then by construction the desired $B_{i,1},B_{i,2}$ exist or $\gamma([a_{i-1},a_i])$ is disjoint with $Y$, in which case we just take $B_{1,i}=B_{2,i}$ to be some connected open neighbourhood of $\gamma([a_{i-1},a_i])$ that is disjoint with $Y$. That the desired $\gamma_i$ exist follows from the fact that $Y$ has empty interior.\\

We now prove that $e(\gamma)$ does not depend on the choices made. First of all suppose we change the choice of $\gamma_i$ for some $i$. Let $\gamma'_i$ be another path in $B_{2,i}\cap B_{2,i+1}$ with $\gamma'_i(0)=\gamma_i(0)=\gamma(a_i)$. Let $e'(a_{i-1},a_i)$ and $e'(a_i,a_{i+1})$ be the resulting elements of $\ZZ/2\ZZ$. By construction $\gamma_i(1)$, $\gamma(a_i)$ and $\gamma'_i(1)$ all belong to the same connected component of $B_{2,i}\cap B_{2,i+1}$. By lemma \ref{agree} the pairs $B_{1,i},B_{2,i}$ and $B_{1,i+1},B_{2,i+1}$ agree on whether or not $\gamma_i(1)$ and $\gamma'_i(1)$ are on the same side of $Y$ or on opposite sides. If they are on the same side of $Y$ then $e(a_{i-1},a_i)=e'(a_{i-1},a_i)$ and $e(a_i,a_{i+1})=e'(a_i,a_{i+1})$. If they are on opposite sides of $Y$ then $e(a_{i-1},a_i)=e'(a_{i-1},a_i)+1$ and $e(a_i,a_{i+1})=e'(a_i,a_{i+1})+1$. But in both cases the sum doesn't change because we work modulo two. So $e(\gamma)$ doesn't depend on the choice of the $\gamma_i$.\newline Note that if we are in the case where $B_{1,i}=B_{2,i}$ is disjoint with $Y$ or $B_{1,i+1}=B_{2,i+1}$ is disjoint with $Y$ then clearly $\gamma_i(1)$ and $\gamma'_i(1)$ are on the same side of $Y$ as they belong to the same connected component of $B_{2,i}\cap B_{2,i+1}$ which is disjoint with $Y$.\\

Next we show that for given $k,a_0,\ldots,a_k$ the result does not depend on the choice of $B_{1,i},B_{2,i}$. Suppose we replace some pair $B_{1,i},B_{2,i}$ by another pair $B'_{1,i},B'_{2,i}$ that still contains $\gamma([a_{i-1},a_i])$. If $i\neq 0$ one can always take a path $\gamma_{i-1}$ whose image is contained in $B_{2,i-1}\cap B_{2,i}\cap B'_{2,i}$. Likewise, if $i\neq k$ one can always take a $\gamma_i$ whose image is contained in $B_{2,i}\cap B'_{2,i}\cap B_{2,i+1}$. By construction $\gamma_{i-1}(1)$ and $\gamma_i(1)$ belong to the same connected component of $B_{2,i}\cap B'_{2,i}$. By lemma \ref{agree} both pairs $B_{1,i},B_{2,i}$ and $B'_{1,i},B'_{2,i}$ agree on whether $\gamma_{i-1}(1)$ and $\gamma_i(1)$ are on the same side of $Y$ or not, so $e(\gamma)$ doesn't change.\\

Finally one has to show that $e(\gamma)$ doesn't depend on the choice of intermediate points $0<a_1<\ldots<a_{k-1}<1$. If one has a second set of intermediate points $0<b_1<\ldots<b_{k'-1}<1$ then one can always throw both sets of points together into a set $0<c_1<\ldots<c_{k''-1}<1$ with
$$\{a_1,\ldots,a_{k-1}\}\cup\{b_1,\ldots,b_{k'-1}\}=\{c_1,\ldots,c_{k''-1}\}.$$
Of course by induction it is enough to show that adding a single extra point to the list $0<a_1<\ldots<a_{k-1}<1$ doesn't change $e(\gamma)$. Suppose we add the point $a$ with $a_{i-1}<a<a_i$ then we have to show that $$e(a_{i-1},a_i)=e(a_{i-1},a)+e(a,a_i).$$
But $\gamma([a_{i-1},a_i])$ is contained in some $B_{2,i}$ and one can choose a path $\gamma':[0,1]\longrightarrow B_{2,i}$ from $a$ to a nearby point not in $Y$ and then all the involved points $\gamma_{i-1}(1)$, $\gamma'(1)$ and $\gamma_{i}(1)$ are contained in $B_{2,i}$, so the result follows.
\end{proof}
\begin{lemma}\label{ii}
Let $X$ be an $n$-dimensional manifold and $Y$ an $n-1$-dimensional submanifold that is a closed subset of $X$.
\begin{itemize}
\item[$(i)$] Let $x$, $y$ and $z$ be points of $X\backslash Y$ and let $\gamma_1$ and $\gamma_2$ be paths from $x$ to $y$ and from $y$ to $z$ respectively. Let $\gamma_3$ be the composition of these two paths (which goes from $x$ to $z$). Then $e(\gamma_1)+e(\gamma_2)=e(\gamma_3)$.
\item[$(ii)$] Let $x,y\in X\backslash Y$ and let $\gamma_1,\gamma_2$ be two paths from $x$ to $y$. If $\gamma_1$ and $\gamma_2$ are path-homotopic then $e(\gamma_1)=e(\gamma_2)$
\end{itemize}
\end{lemma}
\begin{proof}
The proof of $(i)$ is left as an exercise to the reader. We only prove $(ii)$. So let $f:[0,1]\times[0,1]\longrightarrow X$ with $f(s,0)=x$ and $f(s,1)=y$ for all $s\in[0,1]$ and with $\gamma_1(t)=f(0,t)$ and $\gamma_2(t)=f(1,t)$, for all $t\in[0,1]$.
By theorem \ref{localsep} one can cover $Y$ with open sets $B_2$ as in the theorem. The inverse images $f^{-1}(B_2)$ together with $f^{-1}(X\backslash Y)$ form an open cover of $[0,1]\times[0,1]$. We can choose intermediate points $0=a_0<a_1<\ldots<a_k=1$ and $0=b_0<b_1<\ldots<b_{\ell}=1$ such that each $[a_{i-1},a_i]\times[b_{i-1},b_i]$ is contained in such an inverse image. So each $f([a_{i-1},a_i]\times[b_{j-1},b_j])$ is either disjoint with $Y$ or contained in a $B_2$ as in the theorem. In the latter case we call the pair $B_{1,i,j},B_{2,i,j}$. In the former case we just take $B_{1,i,j}=B_{2,i,j}$ some connected open set containing $f([a_{i-1},a_i]\times[b_{j-1},b_j])$ and disjoint with $Y$. For each $i=1,\ldots,k-1$ and $j=1,\ldots,\ell-1$ we choose a path
$$\gamma_{i,j}:[0,1]\longrightarrow B_{2,i,j}\cap B_{2,i+1,j}\cap B_{2,i,j+1}\cap B_{2,i+1,j+1}$$
with $\gamma_{i,j}(0)=f(a_i,b_j)$ and $\gamma_{i,j}(1)\notin Y$.
For $i=0$ (resp.\ $i=k$) same story except the path should go to $B_{2,1,j}\cap B_{2,1,j+1}$ (resp.\ $B_{2,k,j}\cap B_{2,k,j+1}$). For $j=0$, (resp.\ $j=\ell$) we just take $\gamma_{i,j}$ to be the constant path on $x$ (resp.\ $y$). Now for each $i=1,\ldots,k$ and $j=1,\ldots,\ell-1$ we can apply lemma \ref{agree} so that we know that $B_{1,i,j},B_{2,i,j}$ and $B_{1,i,j+1},B_{2,i,j+1}$ agree on whether $\gamma_{i-1,j}(1)$ and $\gamma_{i,j}(1)$ are on the same side of $Y$ or not. This allows us to define for each $i=1,\ldots,k$ and $j=0,\ldots,\ell$ an element $e_{i,j}\in\ZZ/2\ZZ$, expressing whether or not $\gamma_{i-1,j}(1)$ and $\gamma_{i,j}(1)$ are on the same side of $Y$. Likewise, for each $i=0,\ldots,k$ and $j=1,\ldots,\ell$ we can define an element $e'_{i,j}\in\ZZ/2\ZZ$, expressing whether or not $\gamma_{i,j-1}(1)$ and $\gamma_{i,j}(1)$ are on the same side of $Y$. We then have
$$e(\gamma_1)=e'_{0,1}+\ldots+e'_{0,\ell}\text{ and }e(\gamma_2)=e'_{k,1}+\ldots+e'_{k,\ell}.$$
We also have $e_{i,0}=e_{i,k}=0$ for all $i=0,\ldots,k$. To prove the desired equality it is enough to show for all $i=1,\ldots,k$ and $j=1,\ldots,\ell$ that
$$e_{i,j-1}+e'_{i,j}=e'_{i-1,j}+e_{i,j}.$$
But this easily follows from the fact that the four involved points $\gamma_{i-1,j-1}(1)$, $\gamma_{i,j-1}(1)$, $\gamma_{i-1,j}(1)$, $\gamma_{i,j}(1)$, are all contained in $B_{2,i,j}$.
\end{proof}
\begin{proof}[Proof of theorem \ref{twocomponents}]
Let $B_1,B_2$ be balls as in theorem \ref{localsep}. Let $x$ and $y$ be points of $B_2\backslash Y$ that are not in the same connected component of $B_1\backslash Y$. Let $\gamma_1$ be a path in $B_2$ from $x$ to $y$. Then $e(\gamma_1)=1$. By corollary \ref{global} $X\backslash Y$ has at most two connected components. So all we have to prove is that $X\backslash Y$ is not connected. So suppose it is connected, let $\gamma_2$ be a path from $y$ to $x$ that doesn't encounter $Y$. Then $e(\gamma_2)=0$. By lemma \ref{ii} $(i)$ the composed path $\gamma_3$ from $x$ to $x$ satisfies $e(\gamma_3)=1$. Let $G$ be the fundamental group of $X$ based at $x$. By lemma \ref{ii} $(ii)$ we have a map
$$e:G\longmapsto\ZZ/2\ZZ$$
mapping an equivalence class of path-homotopic paths $\gamma$ from $x$ to $x$ to $e(\gamma)$. By lemma \ref{ii} $(i)$ this is a group homomorphism. Now we use the result that first homology group of $X$ is the abelianization of $G$. It follows that this group homomorphism factors through the first homology group (as $\ZZ/2\ZZ$ is Abelian). By our assumption that $\Hom(H_1(X),\ZZ/2\ZZ)=0$ the map is zero. But this contradicts the fact that we constructed a path $\gamma_3$ from $x$ to $x$ with $e(\gamma_3)=1$.
\end{proof}

\end{document}